\documentclass[11pt]{amsart}
\usepackage{tabularx,booktabs}
\usepackage{caption}
\usepackage{amsmath}
\usepackage{amsfonts}

\usepackage{amscd}
\usepackage{amsthm}
\usepackage{amssymb} \usepackage{latexsym}
\usepackage{eufrak}
\usepackage{euscript}
\usepackage{epsfig}
\usepackage{graphics}
\usepackage{array}
\usepackage{enumerate}
\usepackage{dsfont}
\usepackage{color}
\usepackage{wasysym}
\usepackage{hyperref}
\usepackage{pdfsync}

\newcommand{\bel}[1]{\begin{equation}\label{#1}}

\newcommand{\be}{\begin{equation}}

\newcommand{\ba}{\begin{eqnarray}}
\newcommand{\ea}{\end{eqnarray}}

\newcommand{\qe}{\end{equation}}

\newcommand{\Hmm}[1]{\leavevmode{\marginpar{\tiny%
$\hbox to 0mm{\hspace*{-0.5mm}$\leftarrow$\hss}%
\vcenter{\vrule depth 0.1mm height 0.1mm width \the\marginparwidth}%
\hbox to
0mm{\hss$\rightarrow$\hspace*{-0.5mm}}$\\\relax\raggedright #1}}}

\newtheorem{theorem}{Theorem}[section]

\newtheorem{lemma}[theorem]{Lemma}
\newtheorem{corollary}[theorem]{Corollary}
\newtheorem{definition}[theorem]{Definition}

\newtheorem{prop}[theorem]{Proposition}

\makeatletter

\@addtoreset{equation}{section}

\makeatother
\begin{document}

\title[Gradient Estimate of subelliptic harmonic maps with potential]{Gradient Estimate of subelliptic harmonic maps with potential}

\author{Han Luo}
\address{Han Luo: School of Mathematical Sciences, Fudan University, Shanghai 200433, P.R. China }
\email{\href{mailto:19110180023@fudan.edu.cn}{19110180023@fudan.edu.cn}}

\begin{abstract}
In this paper, we investigate subelliptic harmonic maps with potential from noncompact complete sub-Riemannian manifolds corresponding to totally geodesic Riemannian foliations. Under some suitable conditions, we give the gradient estimates of these maps and establish a Liouville type result.
\end{abstract}

\maketitle

\section{Introduction}\label{sec:intro}

Subelliptic harmonic maps were first introduced by Jost and Xu in \cite{10.2307/117803}, which can be regarded as a natural counterpart of harmonic maps in the realm of sub-Riemannian geometry.  Barletta et al. \cite{MR1871387} studied pseudo-harmonic maps from pseudoconvex CR manifolds, which are actually subelliptic harmonic maps defined with respect to the Webster metrics. Dong \cite{MR4236537} obtained Eells-Sampson type results for subelliptic harmonic maps in some general cases. On the other hand, to study some important physics phenomena, the harmonic maps with potential were introduced in \cite{MR1433176}. Some existence results for harmonic maps with potential can be found in \cite{MR1800592} and \cite{MR1680678}.

Suppose that $(M, H, g_H)$ is a sub-Riemannian manifold with a smooth measure $d\mu$ and $(N, h)$ is a Riemannian manifold. Given a function $G \in C^{\infty}(N)$, we consider the following energy functional
\begin{equation}\label{eq:energy functional}
E_G(u)=\frac{1}{2}\int_M [\vert \mathrm{d}u_H\vert^2-2G(u)]\,\mathrm{d}\mu
\end{equation}
where $u: M \to N$ is a smooth map and $\mathrm{d}u_H=\mathrm{d}u\vert_H$. A smooth map $u$ : $(M, H, g_H)\to (N, h)$ is referred to as a subelliptic harmonic map with potential $G$ if it is a critical point of (\ref{eq:energy functional}). The subelliptic harmonic maps with potential can be viewed as a generalization of both harmonic maps with potential and subelliptic harmonic maps. Some Eells-Sampson type existence results of such maps have been achieved in \cite{dong2022subelliptic}.

Many Liouville type theorems have been established in both Riemannian geometry and sub-Riemannian geometry. One may often derive Liouville type results directly from gradient estimates. At first, Yau \cite{MR431040} showed that on a complete Riemannian manifold with nonnegative Ricci curvature, there doesn't exist any nontrivial harmonic function bounded from one side. Later, Cheng \cite{MR573431} proved the Liouville theorem for harmonic maps whose images satisfy the sublinear growth condition when the source manifolds have nonnegative Ricci curvature and the target manifolds have nonpositive sectional curvature. Chen \cite{MR1667241} gave Liouville theorems for harmonic maps with potential whose images are contained in a regular ball on a manifold with sectional curvature bounded from above or a horoball on a Cartan-Hadamard manifold. Ren \cite{MR4045310} achieved a Liouville type results for positive pseudo-harmonic functions on complete pseudo-Hermitian manifolds. Chong et al. \cite{MR4167256} established a Liouville theorem for pseudo-harmonic maps which generalizes the one for harmonic maps by Choi \cite{MR647905}. Zou \cite{zou21} got a gradient estimate of the subelliptic harmonic maps when the source manifolds are step-$2$ complete totally geodesic Riemannian foliations. In this paper, we aim to give gradient estimates of subelliptic harmonic maps with potential, and establish Liouville type results for them.

For a Riemannian foliation $(M, g; \mathfrak{F})$ with a bundle-like metric $g$, we define $H = (T\mathfrak{F})^{\perp}$ as the horizontal subbundle of the foliation $\mathfrak{F}$ with respect to $g$ and denote $g_H=g\vert_H$. It is easy to check that if $H$ is bracket generating for $TM$, then we have a sub-Riemannian manifold $(M, H, g_H; g)$ corresponding to $(M, g; \mathfrak{F})$. In this paper, the subbundle $H$ is always required to have the bracket generating property for $TM$. We call $M$ a step-$r$ sub-Riemannian manifold if sections of $H$ together with their Lie brackets up to order $r$ spans $T_x M$ at each point $x$. Henceforth, we always assume that the source manifold $(M^{m+d},H,g_H;g)$ is a complete step-$r$ sub-Riemannian manifold whose sub-Riemannian structure comes from a totally geodesic Riemannian foliation and the target manifold $(N,h)$ is a complete Riemannian manifold.

Suppose the sectional curvature of $(N,h)$ is bounded above by $\bar K$, where $\bar K\ge 0$ is a constant. Let $\bar B_q(\tau)$ denote a geodesic ball of radius $\tau<\pi/2\sqrt{\bar K}$ centered at $q\in N$, which lies inside the cut locus of $q$. Set
\begin{equation*}
\phi(t)=\begin{cases}
(1-\cos\sqrt{\bar K}t)/\bar K,\qquad &\bar K>0 \\
t^2/2,\qquad &\bar K=0
\end{cases}
\end{equation*}
and
\begin{equation*}
\psi(\cdot)=\phi\circ\rho(\cdot)
\end{equation*}
where $\rho(\cdot)=d_N(\cdot,q)$ denotes the Riemannian distance from $q$ in $N$. Let $M_H$ and $m_H$ be constants such that
\begin{equation*}
\text{Hess}\,G(Y,Y)\le M_G\langle Y,Y\rangle_h, \qquad\forall\, Y\in T_y N, y\in \bar B_q(\tau)
\end{equation*}
and
\begin{equation*}
m_G=\underset{y\in\bar B_q(\tau)}{sup} \langle \widetilde\nabla \psi,\widetilde\nabla G\rangle_h(y)
\end{equation*}
where $\widetilde\nabla$ is the Riemannian connection on $(N, h)$, $\text{Hess}\,G$ is the Hessian matrix of $G$ with respect to $\widetilde\nabla$. Choosing a constant $\delta>0$ such that $\tau<\delta<\pi/2\sqrt{\bar K}$, we set
\begin{equation*}
\beta_1=\cos\sqrt{\bar K}\tau-\cos\sqrt{\bar K}\delta,\quad\beta_2=1-\cos\sqrt{\bar K}\delta
\end{equation*}
and
\begin{equation}\label{eq:definition of f}
f=\beta_2-\bar K\psi.
\end{equation}
It is easy to see that
\begin{equation}\label{eq:bound of f}
0<\beta_1<f<\beta_2<1.
\end{equation}
Let $B_p(R)$ denote the Riemannian geodesic ball of radius $R$ centered at $p \in M$, we have the following
\begin{theorem}\label{thm:gradient estimate of step r}
Let $(M^{m+d},H,g_H;g)$ be a noncompact complete sub-Riemannian manifold corresponding to a totally
geodesic Riemannian foliation with
\begin{equation*}
\quad Ric_H \ge- k_1,\quad and \quad \vert T\vert,\,\vert div_H T\vert\le k_2
\end{equation*}
where $Ric_H$ and $T$ are the horizontal Ricci curvature and the torsion of the generalized Bott connection $\nabla^{\mathfrak{B}}$ respectively (see Sect. \ref{sec:preli} for the detailed definitions), $div_H T$ is the horizontal divergence of $T$, $k_1,k_2\ge0$ are constants. Let $N$, $G$ and $\bar B_q(\tau)$ be as above. Suppose $u:M\to\bar B_q(\tau)\subset N$ is a subelliptic harmonic map with potential $G$. If
\begin{equation*}
k_1+M_G+\frac{2k_2^2}{(1-\beta_2)\beta_1}+k_2+\bar K\text{max}\left(\frac{m_G}{\beta_1},\frac{m_G}{\beta_2}\right)< 0,
\end{equation*}
then $u$ must be a constant map and $u(M)\subset\Sigma_G:=\{y\in N|\widetilde\nabla G(y)=0\}$. Otherwise, we have the
following estimate
\begin{equation*}
\underset{B_p(R)}{\text{sup}} e_H(u)\le C\left(k_1+M_G+\frac{4k_2^2}{(1-\beta_2)\beta_1}+k_2+\bar K\text{max}\left(\frac{m_G}{\beta_1},\frac{m_G}{\beta_2}\right)+\frac{1+R}{R^2}\right)
\end{equation*}
where $C(m,k_1,k_2,\bar K,\tau)$ is a positive constant.
\end{theorem}

If $H$ is 2-step bracket generating for $TM$, adding some suitable condition on $T$, we may get a better estimate as follows.

\begin{theorem}\label{thm:gradient estimate of step 2}
Let $(M^{m+d},H,g_H;g)$ be a step-$2$ noncompact complete sub-Riemannian manifold corresponding to a totally
geodesic Riemannian foliation with
\begin{equation*}
\quad Ric_H \ge- k_1,\quad  \vert T\vert\le k_2,\quad and \quad div_H T=0
\end{equation*}
where $k_1,k_2\ge0$ are constants. Let $N$, $G$ and $\bar B_q(\tau)$ be as above. Suppose $u:M\to\bar B_q(\tau)\subset N$ is a subelliptic harmonic map with potential $G$. We have the
following estimate
\begin{equation*}
\underset{B_p(R)}{\text{sup}} e_H(u)\le C(m,k_1,k_2,\bar K,\tau)\left(k_1+M_G+\frac{1+R}{R^2}\right)
\end{equation*}
where $C(m,k_1,k_2,\bar K,\tau)$ is a positive constant.
\end{theorem}

A direct application of Theorem \ref{thm:gradient estimate of step 2} is the following Liouville type result.

\begin{corollary}\label{thm:Liouville Theorem}
Let $(M^{m+d},H,g_H;g)$ be a step-$2$ noncompact complete sub-Riemannian manifold corresponding to a totally
geodesic Riemannian foliation with
\begin{equation*}
\quad Ric_H \ge 0,\quad  \vert T\vert\le k_2,\quad and \quad div_H T=0
\end{equation*}
where $k_2\ge0$ is a constant. Let $N$ and $\bar B_q(\tau)$ be as above, and let $G$ be a smooth function with $\text{Hess}\, G\le 0$ on $N$. If $u:M\to\bar B_q(\tau)\subset N$ is a subelliptic harmonic map with potential $G$, then $u$ must be a constant map and $u(M)\subset\Sigma_G$.
\end{corollary}

Note that Theorem \ref{thm:gradient estimate of step 2} and Corollary \ref{thm:Liouville Theorem} hold true in the case that the source manifolds are Sasakian manifolds. Since Sasakian manifolds can be seen as  step-$2$ sub-Riemannian manifolds and the Reeb foliation on any Sasakian manifold is actually totally geodesic with the bundle-like metric and of Yang-Mills type, that is, $div_H T=0$ (cf.\cite{MR3587668}).   

\section{Preliminaries}\label{sec:preli}

A sub-Riemannian manifold is defined as a triple $(M, H, g_H)$, where $M$ is a connected smooth manifold, $H$ is a subbundle bracket generating for $TM$, and $g_H$ is a smooth fiberwise metric on $H$. According to \cite{MR862049}, there always exists a Riemannian metric $g$ on $M$ such that $g\vert_H=g_H$, where $g$ is referred to as a Riemannian extension of $g_H$. From now on, we always fix a Riemannian extension $g$ on the sub-Riemannian manifold $(M,H,g_H)$, and consider the quadruple $(M,H,g_H,g)$. According to $g$, the tangent bundle $TM$ has the following orthogonal decomposition:
\begin{equation}\label{eq:orthogonal decomposition}
TM=H\oplus V.
\end{equation}
which induces the projections $\pi_H : TM \to H$ and $\pi_V : TM \to V$. Such $H$ and $V$ are
called the horizontal distribution and the vertical distribution, respectively.

On a sub-Riemannian manifold $(M^{m+d},H,g_H;g)$, there are two canonical distances.
One is the Carnot-Carath\'eodory distance $d_{CC}$ of the sub-Riemannian structure $(H, g_H)$ (cf.\cite{MR862049}), and the other is the Riemannian distance $d_{Rm}$ of $g$. It was proved in \cite{MR793239} that $d_{CC}$ and $d_{Rm}$ induce the same topology, if $H$ is bracket generating for $TM$. Note that the Riemannian distance has better regularity and its variational theory is well studied in Riemannian geometry. In this paper, we restrict our discussion to the Riemannian distance $d_{Rm}$ on $M$ and $B_p(R)$ which is the Riemannian geodesic ball of radius $R$ centered at $p \in M$.

The generalized Bott connection on sub-Riemannian manifold is given by (cf.\cite{MR3587668},\cite{baudoin2015logsobolev},\cite{MR4236537})
\begin{equation}\label{eq:bott connection}
\nabla^{\mathfrak B}_X Y=\begin{cases}
\pi_H(\nabla^R_X Y),\qquad &X,Y\in\Gamma(H) \\
\pi_H([X,Y]),\qquad &X\in\Gamma(V),Y\in\Gamma(H) \\
\pi_V([X,Y]),\qquad &X\in\Gamma(H),Y\in\Gamma(V) \\
\pi_V(\nabla^R_X Y),\qquad &X,Y\in\Gamma(V)
\end{cases}
\end{equation}
where $\nabla^R$ denotes the Riemannian connection of $g$. It is convenient for computations on sub-Riemannian manifold by using the above connection, since $\nabla^{\mathfrak B}$ preserves the decomposition (\ref{eq:orthogonal decomposition}). However, in general, $\nabla^{\mathfrak{B}}$ does not preserve the Riemannian metric $g$. The torsion and the the curvature of $\nabla^{\mathfrak{B}}$ can be expressed as (c.f.\cite{MR4236537})
\begin{equation*}
\begin{split}
T(X,Y)&=\nabla^{\mathfrak{B}}_X Y-\nabla^{\mathfrak{B}}_Y X-[X,Y]\\
&=-\pi_V([\pi_H(X),\pi_H(Y)])-\pi_H([\pi_V(X),\pi_V(Y)])
\end{split}
\end{equation*}
and
\begin{equation*}
R(X, Y )Z = \nabla^{\mathfrak{B}}_X\nabla^{\mathfrak{B}}_Y Z-\nabla^{\mathfrak{B}}_Y \nabla^{\mathfrak{B}}_X Z-\nabla^{\mathfrak{B}}_{[X,Y]}Z.
\end{equation*}
respectively, for $X, Y \in \Gamma(T M)$.
Choose a local orthonormal frame field $\{e_A\}^{m+d}_{A=1}$ on an open domain $\Omega$ of $(M,g)$ such that $span\{e_i\}^{m}_{i=1}=H$, and thus $span\{e_\alpha \}^{m+d}_{\alpha=m+1}=V$. We call such a frame field an adapted frame field for $(M,H,g_H;g)$. Denote its dual frame field by ${\lbrace\omega^A\rbrace}_{A=1}^{m+d}$. From now on, we will make use of the following convention on the ranges of induces in $M$:
\begin{equation*}
\begin{split}
1\le A,B,C,\ldots,&\le m+d;\quad 1\le i,j,k,\ldots,\le m;\\
m+1&\le\alpha,\beta,\gamma,\ldots,\le m+d,
\end{split}
\end{equation*}
and the Einstein summation convention.
Using the frame field, the torsion components and the curvature components are given
by
\begin{equation*}
T^A_{BC}=\langle T(e_B,e_C),e_A\rangle
\end{equation*}
and
\begin{equation*}
R^A_{BCD}=\langle R(e_C, e_D)e_B, e_A\rangle
\end{equation*}
respectively.
Consequently, the horizontal divergence of $T$ is
\begin{equation*}
div_H T (X) = \text{trace}_H (\nabla^{\mathfrak{B}}_{e_i} T)(X,e_i)\quad \text{for} \quad X \in \Gamma(TM),
\end{equation*}
and horizontal Ricci tensor is
\begin{equation*}
Ric_H (X) = \sum\limits^{m}_{i=1} R(X, e_i)e_i \quad \text{for} \quad X \in \Gamma(TM).
\end{equation*}
The horizontal gradient of a smooth function $f$ is defined by
\begin{equation*}
\nabla^H f = \pi_H (\nabla^{\mathfrak{B}}f).
\end{equation*}
As we know, the divergence of a vector field $X$ on $M$ is given by
\begin{equation*}
div_g X=\sum\limits^{m+d}_{A=1}\lbrace e_A\langle X,e_A\rangle - \langle X,\nabla^R_{e_A}e_A\rangle\rbrace.
\end{equation*}
Then one can define sub-Laplacian of a function $f$ on $(M,H,g_H;g)$ as
\begin{equation}\label{eq:sub laplacian definition}
\Delta_H f=div_g(\nabla^H f)= \text{trace}(\nabla^{\mathfrak{B}}\mathrm{d}f\vert_{H\times H})-\zeta f
\end{equation}
where $\zeta=\pi_H(\sum\limits_{\alpha}\nabla^R_{e_{\alpha}}e_{\alpha})$ is referred to as the mean curvature vector field of the vertical distribution $V$.

We call $M$ a step-$r$ sub-Riemannian manifold if sections of $H$ together with their Lie brackets up to order $r$ spans $T_x M$ at each point $x$. Let $S(V)=\{v\in V :\Vert v \Vert_g=1\}$ denote the unit sphere bundle of the vertical bundle $V$. For any $v\in S(V)$, the $v$-component of $T(\cdot,\cdot)$ is given by $T^v(\cdot,\cdot)=\langle T(\cdot,\cdot),v\rangle$. Then we have a smooth function $\eta(v)=\frac 1 2\Vert T^v\Vert^2_g:S(V)\to \mathbb{R}$. If $(M^{m+d},H,g_H;g)$ is a step-2 sub-Riemannian manifold, we have the following
\begin{lemma}(\cite[Lemma 6.6]{MR4236537})\label{thm:eta}
H is $2$-step bracket generating if and only if $\eta(v)>0$ for each $v\in S(V)$.
\end{lemma}
The above lemma suggests that $\eta(v)$ achieves a positive minimal value $\eta_{min}$ on $S(V)$.

For a Riemannian foliation $(M, g; \mathfrak{F})$, let $H$ be the orthogonal component of $V = T\mathfrak{F}$ and $g_H$ be the restriction of $g$ to $H$. Then we have a sub-Riemannian manifold $(M, H, g_H; g)$ corresponding to $(M, g; \mathfrak{F})$. From \cite[Lemma 1.4.1]{MR2500106}, one may find that $R^j_{i\alpha k}=0$ and $Ric_H$ is symmetric. Furthermore, if $\mathfrak{F}$ is totally geodesic, then the vector field $\zeta = 0$ and $\nabla g$ = 0. Readers may refer to
\cite{MR3587668}, \cite{MR4236537} for details.

We also need the estimate of sub-Laplacian of Riemannian distance on $M$. Although sub-Laplacian comparison theorems have been investigated for some special cases in \cite{MR3299291}
\cite{baudoin2018sublaplacian}\cite{MR3935276}\cite{MR3853931}\cite{MR4167256}, there is no satisfactory comparison theorem for a sub-Riemannian manifold in general up to now. Fortunately, when $M$ is a complete totally geodesic Riemannian foliation, according to \cite[Theorem 3.1]{HY22}, one can get the following result:
\begin{equation}\label{eq:sub-Laplacian of Riemannian distance}
\Delta_H r\le C (1+\frac{1}{r}),\qquad on \quad B_p(R)\backslash Cut(p)
\end{equation}
where $C(m,k_1,k_2)$ is a positive constant, $r$ is the Riemannian distance from the fixed point $p$ and $Cut(p)$ is the cut locus of $p$.

Let $(N,h)$ be a Riemannian manifold with the Riemannian connection $\widetilde\nabla$ and
the Riemannian curvature $\widetilde R$. We choose an orthonormal frame field $\lbrace\widetilde e_I\rbrace_{I=1,\ldots,n}$ in $(N,h)$ and make use of the following convention on the ranges of indices in $N$:
\begin{equation*}
I,J,K=1,\ldots,n.
\end{equation*}

For a smooth map $u:M\to N$, in terms of the frame fields in $M$ and $N$, the differential $\mathrm{d}u$ and the second fundamental form $\beta$ can be written as
\begin{equation*}
\mathrm{d}u=u_A^I\omega^A\otimes\widetilde e_I,
\end{equation*}
and
\begin{equation*}
\beta=u^I_{AB}\omega^A\otimes\omega^B\otimes\widetilde e_I
\end{equation*}
respectively.
Apart from the differential $\mathrm{d}u$, we also introduce two partial differentials $ \mathrm{d}u_H= \mathrm{d}u\vert_H \in\Gamma(H^*\otimes u^{-1}TN)$ and $\mathrm{d}u_V= \mathrm{d}u\vert_V \in\Gamma(V^*\otimes u^{-1}TN)$. Then, we get
\begin{equation*}
\vert\mathrm{d}u_H\vert^2=(u^I_i)^2,\quad \vert \mathrm{d}u_V\vert^2=(u^I_{\alpha})^2,\quad \vert\mathrm{d}u\vert^2=(u^I_A)^2.
\end{equation*}
Define
\begin{equation*}
e_H(u)=\frac 1 2\vert\mathrm{d}u_H\vert^2,\quad e_V(u)=\frac 1 2\vert  \mathrm{d}u_V\vert^2,\quad e(u)=\frac 1 2\vert\mathrm{d}u\vert^2.
\end{equation*}
For any potential function $G\in C^{\infty}(N)$, we introduce the following energy:
\begin{equation}\label{eq:energy functional 2}
E_G(u)=\int_M [e_H(u)-G(u)]\,\mathrm{d}v_g.
\end{equation}
The energy $E_G(u)$ is called horizontal energy with potential $G$.
\begin{definition}(\cite{dong2022subelliptic})
A map $u:(M,H,g_H;g)\to (N,h)$ is called a subelliptic harmonic map with potential $G$ if it is a critical point of the energy $E_G(u)$.
\end{definition}
The Euler-Lagrange equation of (\ref{eq:energy functional 2}) is
\begin{equation}\label{eq:tension field}
\tau_G(u)=\tau_H(u)+(\widetilde\nabla G)(u)=\beta(e_i,e_i)-\mathrm{d}u(\zeta)+(\widetilde\nabla G)(u)=0
\end{equation}
where $\tau_H(u)$ is the subelliptic tension field associated with the horizontal energy (c.f.\cite{MR4236537}).
Therefore, we have the following equivalent characterization of subelliptic harmonic maps with potential $G$.
\begin{prop}(\cite{dong2022subelliptic})
A map $u:(M,H,g_H;G)\to (N,h)$ is a subelliptic harmonic map with potential $G$ if and only if it satisfies the Euler-Lagrange equation
\begin{equation*}
\tau_G(u)=0.
\end{equation*}
We call $\tau_G(u)$ the subelliptic tension field of $u$ with potential $G$.
\end{prop}

For our purpose, we need the following Bochner type inequality for $e(u)$.
\begin{lemma}\label{thm:subelliptic bochner formula}
Let $(M^{m+d},H,g_H;g)$ be a complete sub-Riemannian manifold which is corresponding to a totally geodesic Riemannian foliation with
\begin{equation*}
Ric_H \ge- k_1,\quad and \quad \vert T\vert,\,\vert div_H T\vert\le k_2.
\end{equation*}
Suppose $(N,h)$ is a complete Riemannian manifold with $K_N\le \bar K$ and $u:M\to N$ is a subelliptic harmonic map with potential $G$, then one has
\begin{equation}\label{eq:step r bochner formula}
\Delta_H e(u)\ge (1-k_2\epsilon)(u^I_{Ak})^2-\left(2M_G+2k_1+\frac{(2+\epsilon)k_2}{\epsilon}\right)e(u)-4\bar K e(u)\cdot e_H(u)
\end{equation}
for any given $\epsilon>0$.
In particular, if H is 2-step bracket generating for $TM$ and $div_H T=0$, we have
\begin{equation}\label{eq:step 2 horizontal bochner formula}
\begin{split}
\Delta_H e_H\ge
&(1-\epsilon_1)(u^I_{ik})^2-\left(2k_1+\frac{C}{\epsilon_2}+2M_G\right)e_H(u)\\
&+\frac 1 2\epsilon_1\eta_{min} e_V(u)-C\epsilon_2 (u^I_{\alpha k})^2-4\bar K e^2_H(u)
\end{split}
\end{equation}
and
\begin{equation}\label{eq:step 2 vertical bochner formula}
\Delta_H e_V\ge
(u^I_{\alpha k})^2-2M_G e_V(u)-4\bar K e_H(u)\cdot e_V(u)
\end{equation}
for any given $0<\epsilon_1<1$ and $\epsilon_2>0$, where $C(k_2)$ is a positive constant only depending on $T$.
\end{lemma}

\begin{proof}[Proof of Lemma \ref{thm:subelliptic bochner formula}]
Denote covariant derivatives of $\zeta^k$ and $T^{A}_{BC}$ by $\zeta^k_{,A}$ and $T^{A}_{BC,D}$ respectively.
From \cite{MR4236537}, we know
\begin{equation*}
\begin{split}
\Delta_H e_H(u)&=(u^I_{ik})^2+u^I_i\tau^I_{H,i}+u^I_i\zeta^k_{,k}u^I_k+\zeta^k u^I_i u^I_{\alpha}T^{\alpha}_{ki} \\
&\quad+u^I_i u^I_j R^j_{kik}+2u^I_{i}u^I_{\alpha k}T^{\alpha}_{ik}-u^I_i u^K_k\widetilde R^I_{KJL}u^J_i u^L_k+u^I_i u^I_{\alpha}T^{\alpha}_{ik,k}
\end{split}
\end{equation*}
and
\begin{equation*}
\begin{split}
\Delta_H e_V(u)&=(u^I_{\alpha k})^2+u^I_{\alpha}\tau^I_{H,\alpha}+u^I_{\alpha}\zeta^k_{,\alpha}u^I_k+u^I_{\alpha}u^I_j R^j_{k\alpha k}-u^I_{\alpha}u^K_k\widetilde R^I_{KJL}u^J_{\alpha}u^L_k.
\end{split}
\end{equation*}
Since $\tau^I=\tau^I_H+[(\widetilde\nabla G)(u)]^I$, then
\begin{equation*}
\begin{split}
\Delta_H e_H(u)&=(u^I_{ik})^2+u^I_i\tau^I_{H,i}-u^I_i G_{IJ}u^J_i+u^I_i\zeta^k_{,k}u^I_k+\zeta^k u^I_i u^I_{\alpha}T^{\alpha}_{ki} \\
&\quad+u^I_i u^I_j R^j_{kik}+2u^I_{i}u^I_{\alpha k}T^{\alpha}_{ik}-u^I_i u^K_k\widetilde R^I_{KJL}u^J_i u^L_k+u^I_i u^I_{\alpha}T^{\alpha}_{ik,k}
\end{split}
\end{equation*}
and
\begin{equation*}
\begin{split}
\Delta_H e_V(u)=&(u^I_{\alpha k})^2+u^I_{\alpha}\tau^I_{H,\alpha}-u^I_{\alpha}G_{IJ}u^J_{\alpha}\\
&+u^I_{\alpha}\zeta^k_{,\alpha}u^I_k+u^I_{\alpha}u^I_j R^j_{k\alpha k}-u^I_{\alpha}u^K_k\widetilde R^I_{KJL}u^J_{\alpha}u^L_k.
\end{split}
\end{equation*}
where $\text{Hess}\,G=(G_{IJ})$.
Since $\zeta=0$ and $R^j_{k\alpha k}=0$, we get
\begin{equation*}
u_i^I \zeta^k_{,i}u^I_k+\zeta^ku_i^Iu^I_{\alpha}T^{\alpha}_{ki}=0
\end{equation*}
and
\begin{equation*}
u^I_{\alpha}\zeta^k
_{,\alpha}u^I_k+u^I_{\alpha}u_j^IR^j_{k\alpha k}=0.
\end{equation*}
For any given $\epsilon>0$, by Schwarz inequality, we obtain
\begin{equation*}
\begin{split}
&u^I_i G_{IJ}u^J_i+u^I_{\alpha}G_{IJ}u^J_{\alpha} \le 2 M_G\cdot e(u),\\
&u_i^Iu_j^IR_{kik}^j\ge-2k_1 e_H(u)\ge-2k_1 e(u),\\
&2u_i^Iu^I_{\alpha k}T^{\alpha}_{ik}+u_i^Iu^I_{\alpha}T^{\alpha}_{ik,k}\ge-\frac{(2+\epsilon)k_2}{\epsilon}e_H(u)-k_2e_V(u)-k_2\epsilon(u^I_{\alpha k})^2 \\
&\hspace{3.75cm}\ge-\frac{(2+\epsilon)k_2}{\epsilon}e(u)-k_2\epsilon(u^I_{\alpha k})^2,\\
&u_i^Iu_k^K \widetilde R^I_{KJL}u_i^Ju_k^L+u^I_{\alpha}u_k^K \widetilde R^I_{KJL}u_{\alpha}^Ju_k^L \le 4\bar K e(u)\cdot e_H(u).
\end{split}
\end{equation*}
The above estimates give (\ref{eq:step r bochner formula}).
Due to \cite[Section 4]{MR4236537}, we have the following equality
\begin{equation}\label{eq:torsion relation}
u^I_{ij}-u^I_{ji}=u^I_{\alpha}T^{\alpha}_{ij}.
\end{equation}
When $H$ is $2$-step bracket generating for $TM$, in terms of Lemma \ref{thm:eta} and (\ref{eq:torsion relation}), we know that
\begin{equation*}
\begin{split}
(u^I_{ik})^2
\ge&\frac 1  2\sum\limits_I\sum\limits_{i<j}((u^I_{ij}+u^I_{ji})^2+(u^I_{ij}-u^I_{ji})^2) \\
\ge&\frac 1 2\sum\limits_I\sum\limits_{\alpha}\sum\limits_{i<j}(u^I_{\alpha})^2(T^{\alpha}_{ij})^2 \\
=&\frac 1  2\sum\limits_I\sum\limits_{\alpha}(u^I_{\alpha})^2\eta(e_{\alpha}) \\
\ge&\eta_{min}e_V(u).
\end{split}
\end{equation*}
Using Schwarz inequality and $div_H T=0$, we get (\ref{eq:step 2 horizontal bochner formula}) and (\ref{eq:step 2 vertical bochner formula}).
\end{proof}

\section{Proof of the theorem}\label{sec:proof}

\begin{proof}[Proof of Theorem \ref{thm:gradient estimate of step r}]
Set
\begin{equation*}
A (x) =\frac{e(u)(x)}{f^2(u)(x)}
\end{equation*}
where $f$ is defined in (\ref{eq:definition of f}).
To simplify the notations, we write $ e=e(u), e_H=e_H(u), f= f(u)$. By a direct computation, we have
\begin{equation}\label{eq:sub laplacian A}
\Delta_H A=\frac{\Delta_H e}{f^2}-\frac{4\nabla^H e\cdot\nabla^H f}{f^3}-\frac{2e\Delta_H f}{f^3}+\frac{6e\vert\nabla^H f\vert^2}{f^4}.
\end{equation}
In terms of (\ref{eq:definition of f}) and Riemannian comparison theorem, we get
\begin{equation}\label{eq:upper bound of sub laplacian f}
\begin{split}
\Delta_H f&=\text{Hess}(f)(du_H,du_H)+\langle(\widetilde\nabla f)(u),\tau_H(u) \rangle \\
&=\text{Hess}(f)(du_H,du_H)-\langle(\widetilde\nabla f)(u),(\widetilde\nabla G)(u) \rangle \\
&=-\bar K \text{Hess}(\psi)(du_H,du_H)+\bar K\langle(\widetilde\nabla \psi)(u),(\widetilde\nabla G)(u) \rangle \\
&\le-2\bar K(1-\bar K\psi\circ u)e_H+\bar K m_G \\
&=-2\bar K(1-\beta_2+f)e_H+\bar K m_G.
\end{split}
\end{equation}

Applying (\ref{eq:step r bochner formula}) and (\ref{eq:upper bound of sub laplacian f}) to (\ref{eq:sub laplacian A}), we have
\begin{equation}\label{eq:crude sub laplacian A inequality}
\begin{split}
\Delta_H A\ge &-\frac{2k_1e}{f^2}-\frac{4\bar K e\cdot e_H}{f^2}-\frac{2 e}{f^2}M_G-\frac{\frac{(2+\epsilon)k_2}{\epsilon} e}{f^2}+\frac{4\bar K(1-\beta_2+f) e\cdot e_H}{f^3} \\
&-\frac{2\bar K e}{f^3}m_G +[\frac{(1-k_2\epsilon)(u^I_{Ak})^2}{f^2}-\frac{4\nabla^H e\cdot\nabla^H f}{f^3}+\frac{6e\vert\nabla^H f\vert^2}{f^4}]
\end{split}
\end{equation}
where $\epsilon$ is a positive constant to be decided.
Note that
\begin{equation}\label{eq:upper bound of horizontal gradient energy}
\nabla^H e=u^I_{A}u^I_{Ak}\le\sqrt{(u^I_{A}u^I_{Ak})^2}=\vert u^I_{Ak}\vert\sqrt{2e}
\end{equation}
and
\begin{equation}\label{eq:upper bound of horizontal gradient f}
\vert\nabla^H f\vert^2\le 2\vert\widetilde\nabla f\vert^2e_H\le 2\bar K e_H.
\end{equation}
By using (\ref{eq:upper bound of horizontal gradient energy}) and (\ref{eq:upper bound of horizontal gradient f}), the last term on the right-hand side of (\ref{eq:crude sub laplacian A inequality}) becomes
\begin{equation*}
\begin{split}
[\cdots]=&\frac{(1-k_2\epsilon)(u^I_{Ak})^2}{f^2}-\frac{2(1-k_2\epsilon)\nabla^H e\cdot\nabla^H f}{f^3}+\frac{2(1-k_2\epsilon)e\vert\nabla^H f\vert^2}{f^4} \\
&-\frac{2(1+k_2\epsilon)\nabla^H e\cdot\nabla^H f}{f^3}+\frac{(4+2k_2\epsilon)e\vert\nabla^H f\vert^2}{f^4} \\
\ge &\frac{(1-k_2\epsilon)}{f^2}\left[(u^I_{Ak})^2-2\vert u^I_{Ak}\vert\sqrt{2e}\cdot \frac{\vert\nabla^H f\vert}{f^3}+2e\frac{\vert\nabla^H f\vert^2}{f^4}\right]\\
&-\frac{2(1+k_2\epsilon)\nabla^H A\cdot\nabla^H f}{f}-\frac{2k_2\epsilon e\vert\nabla^H f\vert^2}{f^4}\\
\ge &-\frac{2(1+k_2\epsilon)\nabla^H A\cdot\nabla^H f}{f}-\frac{4 k_2\bar K\epsilon e\cdot e_H}{f^4}.
\end{split}
\end{equation*}
Therefore, we get
\begin{equation*}
\begin{split}
\Delta_H A\ge &\frac{4\bar K(1-\beta_2)e\cdot e_H}{f^3}-\frac{4k_2\bar K\epsilon e\cdot e_H}{f^4}-\frac{2(1+k_2\epsilon)\nabla^H A\cdot\nabla^H f}{f} \\
&-2\left(k_1+M_G+\frac{(2+\epsilon)k_2}{\epsilon}+\frac{\bar K m_G}{f}\right)\frac{e}{f^2} \\
\ge &4\bar K[(1-\beta_2)f-k_2\epsilon]\frac{Ae_H}{f^2}-\frac{2(1+k_2\epsilon)\nabla^H A\cdot\nabla^H f}{f}  \\
&-2\left(k_1+M_G+\frac{(2+\epsilon)k_2}{\epsilon}+\frac{\bar K m_G}{f}\right)A.
\end{split}
\end{equation*}
Choosing $0<\epsilon<\frac{(1-\beta_2)\beta_1}{k_2}$, we have
\begin{equation}\label{eq:sub laplacian A inequality}
\Delta_H A\ge C_0 \frac{Ae_H}{f^2}-2s_G A-\frac{2(1+k_2\epsilon)\nabla^H A\cdot\nabla^H f}{f}
\end{equation}
where $C_0=4\bar K[(1-\beta_2)f-k_2\epsilon]>0$ and $s_G(\epsilon)=k_1+M_G+
\frac{(2+\epsilon)k_2}{\epsilon}+\bar K\text{max}\left(\frac{m_G}{\beta_1},\frac{m_G}{\beta_2}\right)$.
It is easy to see that $s_G(\epsilon)$ is a decreasing function of $\epsilon$ for $\epsilon>0$.

Choose a cut-off function
\begin{equation*}
\varphi\vert_{[0,1]}=1, \varphi\vert_{[2,\infty)}=0,-C'_1\vert\varphi\vert^{\frac{1}{2}}\le\varphi'\le0
\end{equation*}
where $C'_1$ is a positive constant.
Let $\chi(r)=\varphi(\frac{r}{R})$, owing to (\ref{eq:sub-Laplacian of Riemannian distance}), we find that
\begin{equation}\label{eq:bound of horizontal gradient chi}
\frac{\vert\nabla^H\chi\vert^2}{\chi}\le\frac{C_1}{R^2}
\end{equation}
and
\begin{equation}\label{eq:lower bound of subelliptic chi}
\Delta_H\chi\ge-\frac{C_1}{R} \quad in \quad B(2R)\backslash Cut(p)
\end{equation}
where $C_1(m,k_1,k_2)$ is a positive constant.

Let $x_R$ be a maximum point of $\chi A(x)=\chi(r(x))A(x)$ in $B_R(p)$, then at $x_R$, we have
\begin{equation}\label{eq:horizontal gradient at maximum point 1}
\nabla^H(\chi A)(x_R)=\chi\nabla^H A+A\nabla^H \chi=0
\end{equation}
and
\begin{equation}\label{eq:sub laplacian at maximum point 1}
\Delta_H(\chi A)(x_R)\le 0.
\end{equation}
Substituting (\ref{eq:sub laplacian A inequality}), (\ref{eq:bound of horizontal gradient chi}), (\ref{eq:lower bound of subelliptic chi}) and (\ref{eq:horizontal gradient at maximum point 1}) into (\ref{eq:sub laplacian at maximum point 1}) yields, at point $x_R$
\begin{equation*}
\begin{split}
0\ge &A\Delta_H \chi+\chi \Delta_H A+2\nabla^H \chi\cdot\nabla^H A \\
=&A\Delta_H \chi+\chi \Delta_H A-2\frac{\vert\nabla^H\chi\vert^2}{\chi}A \\
\ge&\left(C_0\frac{e_H}{f^2}A-2s_G A-2(1+k_2\epsilon)\frac{\nabla^H A\cdot\nabla^H f}{f}\right)\chi \\
&-\frac{C_1}{R}A-2\frac{C_1}{R^2}A.
\end{split}
\end{equation*}
Multiplying both sides with $\chi$, we obtain
\begin{equation}\label{eq:crude main inequality}
0\ge C_0\frac{e_H}{f^2}A\chi^2-2s_GA\chi^2-\frac{C_2}{R^2}(1+R)A\chi-2(1+k_2\epsilon)\frac{\nabla^H A\cdot\nabla^H f}{f}\chi^2
\end{equation}
where $C_2(m,k_1,k_2)$ is a positive constant.
From (\ref{eq:upper bound of horizontal gradient f}) and (\ref{eq:horizontal gradient at maximum point 1}), we know
\begin{equation}\label{eq:product term}
\begin{split}
-\frac{\chi\nabla^H A\cdot\nabla^H f}{f}&=\frac{A\nabla^H\chi\cdot\nabla^H f}{f} \\
&\ge-A\vert\nabla^H\chi\vert\frac{\vert\nabla^H f\vert}{f} \\
&\ge-A\vert\nabla^H\chi\vert\frac{\sqrt{2\bar K e_H}}{f} \\
&\ge-\frac {C_3}{R}\sqrt{\frac{e_H\chi}{f^2}}A
\end{split}
\end{equation}
where $C_3(m,k_1,k_2,\bar K)$ is a positive constant.
By (\ref{eq:crude main inequality}) and (\ref{eq:product term}), we have
\begin{equation}\label{eq:main inequality 1}
0\ge C_0\frac{e_H\chi}{f^2}\cdot A\chi-2s_G\chi\cdot A\chi-\frac{C_2}{R^2}(1+R)\cdot A\chi-\frac{C_4}{R}\sqrt{\frac{e_H\chi}{f^2}}\cdot A\chi
\end{equation}
where $C_4(m,k_1,k_2,\bar K,\epsilon)$ is a positive constant.
We assume that $A\chi(x_R)>0$, otherwise $u$ is a constant. Multiplying both sides of (\ref{eq:main inequality 1}) with $\frac{1}{A\chi}$, we get
\begin{equation}\label{eq:main inequality 2}
0\ge C_0\frac{e_H\chi}{f^2}-2s_G\chi -\frac{C_2}{R^2}(1+R)-\frac{C_4}{R}\sqrt{\frac{e_H\chi}{f^2}}.
\end{equation}

If $s_G(\frac{(1-\beta_2)\beta_1}{k_2})<0$, that is,
\begin{equation*}
k_1+M_G+\frac{2k_2^2}{(1-\beta_2)\beta_1}+k_2+\bar K\text{max}\left(\frac{m_G}{\beta_1},\frac{m_G}{\beta_2}\right)<0,
\end{equation*}
by the monotonicity of $s_G(\epsilon)$, we can choose a fixed number $0<\epsilon_0<\frac{(1-\beta_2)\beta_1}{k_2}$, such that $s_G(\epsilon_0)\le 0$. Then, we obtain
\begin{equation*}
0\ge C_0\frac{e_H\chi}{f^2}-\frac{C_4}{R}\sqrt{\frac{e_H\chi}{f^2}}-\frac{C_2}{R^2}(1+R).
\end{equation*}
It follows easily that
\begin{equation*}
\frac{e_H\chi}{f^2}\le\frac{C_5}{R^2}(1+R)
\end{equation*}
where $C_5(m,k_1,k_2,\bar K,\epsilon_0,\tau)$ is a positive constant.
Since $f<\beta_2$, we deduce that
\begin{equation*}
\frac{1}{\beta_2^2}\underset{B_p(R)}{\text{sup}} e_H(u)\le \underset{B_p(2R)}{\text{sup}} \frac{e_H\chi(u)}{f^2(u)}\le\frac{C_5}{R^2}(1+R).
\end{equation*}
Letting $R\to \infty$, we conclude that $u$ must be a constant map and therefore $u(M)\subset \Sigma_G$.

If $s_G(\frac{(1-\beta_2)\beta_1}{k_2})\ge0$, then $s_G(\epsilon)>0$ for $0<\epsilon<\frac{(1-\beta_2)\beta_1}{k_2}$.
It follows from (\ref{eq:main inequality 2}) that
\begin{equation*}
0\ge C_0\frac{e_H\chi}{f^2}-\frac{C_4}{R}\sqrt{\frac{e_H\chi}{f^2}}-\left(2s_G+\frac{C_2}{R^2}(1+R)\right).
\end{equation*}
Consequently,
\begin{equation*}
e_H\chi\le C_6(2s_G+\frac{1+R}{R^2})
\end{equation*}
where $C_6(m,k_1,k_2,\bar K,\epsilon,\tau)$ is a positive constant.
Setting $\epsilon=\frac{(1-\beta_2)\beta_1}{2k_2}$, we have the following estimate
\begin{equation*}
\underset{B_p(R)}{\text{sup}} e_H(u)\le C\left(k_1+M_G+\frac{4k_2^2}{(1-\beta_2)\beta_1}+k_2+\bar K\text{max}\left(\frac{m_G}{\beta_1},\frac{m_G}{\beta_2}\right)+\frac{1+R}{R^2}\right)
\end{equation*}
where $C(m,k_1,k_2,\bar K,\tau)$ is a positive constant.
\end{proof}

In order to prove Theorem \ref{thm:gradient estimate of step 2}, we also need the following lemma.

\begin{lemma}(\cite{MR4167256},\cite{zou21})\label{thm:estimate of main auxiliary function}
For any $0<\tau<\frac{\pi}{2\sqrt{\bar K}}$, there exist $\nu\in[1,2), b>\phi(\tau)$ and $\delta>0$ only depending on $\tau$ such that
\begin{equation*}
\nu\frac{\cos(\sqrt{\bar K}t)}{b-\phi(t)}-2\bar K>\delta,\qquad \forall t\in [0,\tau].
\end{equation*}
Furthermore, if $u:M\to\bar B_q(\tau)\subset N$ is a subelliptic harmonic map with potential $G$, then we have
\begin{equation*}
\nu\frac{\Delta_H (\psi\circ u)}{b-\psi\circ u}-4\bar K e_H(u)>2\delta e_H(u).
\end{equation*}
\end{lemma}
\begin{proof}[Proof of Theorem \ref{thm:gradient estimate of step 2}]
Define the following auxiliary function
\begin{equation*}
\Phi_{\mu\chi}=e_H(u)+\mu\chi e_V(u)
\end{equation*}
where $\chi$ is the same cut-off function we use in the proof of Theorem \ref{thm:gradient estimate of step r}, $\mu$ is a constant number to be decided. Applying $\epsilon_2=\frac{\epsilon_1\mu\chi}{C}$ to (\ref{eq:step 2 horizontal bochner formula}) and using (\ref{eq:step 2 vertical bochner formula}), we deduce that
\begin{equation}\label{eq:crude sub laplacian Phi inequality}
\begin{split}
\Delta_H \Phi_{\mu\chi}=&\Delta_H(e_H(u)+\mu\chi e_V(u))\\
\ge&(1-\epsilon_1)(u^I_{ik})^2-\left(2k_1+\frac{C}{\epsilon_2}+2M_G\right)e_H(u)\\
&+\frac 1 2\epsilon_1\eta_{min} e_V(u)-C\epsilon_2 (u^I_{\alpha k})^2-4\bar K e^2_H(u)+\mu\Delta_H\chi e_V(u) \\
&+2\mu\chi_k u^I_{\alpha}u^I_{\alpha k}+\mu\chi(u^I_{\alpha k})^2-2\mu\chi M_G e_V(u)-4\bar K\mu\chi e_H(u)e_V(u) \\
=&(1-\epsilon_1)\left((u^I_{ik})^2+\mu\chi(u^I_{\alpha k})^2\right)+2\mu\chi_k u^I_{\alpha}u^I_{\alpha k}-4\bar K\Phi_{\mu\chi}e_H(u) \\
&+\left(\frac{1}{2}\epsilon_1\eta_{min}+\mu\Delta_H\chi-2\mu\chi M_G\right)e_V(u) \\
&-\left(2k_1+\frac{C}{\epsilon_2}+2M_G\right)e_H(u)
\end{split}
\end{equation}
where $\chi_k$ is k-component of $\nabla^H \chi$ given by $\langle \nabla^H \chi, e_k\rangle$.
By Schwarz inequality, we have
\begin{equation}\label{gradient estimate Phi}
\begin{split}
\vert\nabla^H\Phi_{\mu\chi}\vert^2=&\vert \nabla^H (e_H(u)+\mu\chi e_V(u))\vert^2 \\
\le& \vert u^I_k+\sqrt{\mu\chi}u^I_{\alpha}\vert^2\cdot\vert u^I_{ik}+\sqrt{\mu\chi}u^I_{\alpha k}+\sqrt{\mu}\frac{\nabla^H\chi}{2\sqrt{\chi}}\otimes u^I_{\alpha}\vert^2 \\
=& 2\Phi_{\mu\chi}\left((u^I_{ik})^2+\mu\chi(u^I_{\alpha k})^2+\frac{\mu\vert\nabla^H\chi\vert^2}{2\chi}e_V(u)+\mu\chi_k u^I_{\alpha}u^I_{\alpha k}\right).
\end{split}
\end{equation}
Choosing $0<\epsilon_1<\frac{1}{2}$, from (\ref{gradient estimate Phi}), we have the following estimate
\begin{equation}\label{eq:gradient estimate Phi and chi}
\begin{split}
&(1-\epsilon_1)\left((u^I_{ik})^2+\mu\chi(u^I_{\alpha k})^2\right)+2\mu\chi_k u^I_{\alpha}u^I_{\alpha k}\\
\ge&(1-2\epsilon_1)\left((u^I_{ik})^2+\mu\chi(u^I_{\alpha k})^2\right)+\epsilon_1\mu\chi(u^I_{\alpha k})^2+2\mu\chi_k u^I_{\alpha}u^I_{\alpha k} \\
\ge&\left(\frac{1}{2}-\epsilon_1\right)\frac{\vert\nabla^H\Phi_{\mu\chi}\vert^2}{\Phi_{\mu\chi}}-\left(\frac{1}{2}-\epsilon_1\right)\frac{\mu\vert\nabla^H\chi\vert^2}{\chi}e_V(u)\\
&+(1+2\epsilon_1)\mu\chi_k u^I_{\alpha}u^I_{\alpha k}+\epsilon_1\mu\chi(u^I_{\alpha k})^2\\
\ge&\left(\frac{1}{2}-\epsilon_1\right)\frac{\vert\nabla^H\Phi_{\mu\chi}\vert^2}{\Phi_{\mu\chi}}-\left(\frac{1}{2}-\epsilon_1\right)\frac{\mu\vert\nabla^H\chi\vert^2}{\chi}e_V(u)\\
&-\frac{(1+2\epsilon_1)^2}{4\epsilon_1}\frac{\mu\vert\nabla^H\chi\vert^2}{\chi}e_V(u) \\
\ge&\left(\frac{1}{2}-\epsilon_1\right)\frac{\vert\nabla^H\Phi_{\mu\chi}\vert^2}{\Phi_{\mu\chi}}-4\epsilon_1^{-1}\frac{\mu\vert\nabla^H\chi\vert^2}{\chi}e_V(u),
\end{split}
\end{equation}
since
\begin{equation*}
\frac{1}{2}-\epsilon_1+\frac{(1+2\epsilon_1)^2}{4\epsilon_1}\le \frac{3}{2}+\frac{1}{4}\epsilon_1^{-1}\le 2\epsilon_1^{-1}.
\end{equation*}
Combining (\ref{eq:crude sub laplacian Phi inequality}) with (\ref{eq:gradient estimate Phi and chi}), we get
\begin{equation}\label{eq:sub laplacian Phi inequality}
\begin{split}
\Delta_H \Phi_{\mu\chi}=&\left(\frac{1}{2}-\epsilon_1\right)\frac{\vert\nabla^H\Phi_{\mu\chi}\vert^2}{\Phi_{\mu\chi}}-4\bar K\Phi_{\mu\chi}e_H(u) \\
&+\left(\frac{1}{2}\epsilon_1\eta_{min}+\mu\Delta_H\chi-2\epsilon_1^{-1}\frac{\mu\vert\nabla^H\chi\vert^2}{\chi}-2\mu\chi M_G\right)e_V(u) \\
&-\left(2k_1+\frac{C}{\epsilon_2}+2M_G\right)e_H(u).
\end{split}
\end{equation}
We consider the following function
\begin{equation*}
F_{\mu\chi}=\frac{\Phi_{\mu\chi}}{(b-\psi\circ u)^{\nu}}.
\end{equation*}
Suppose $\chi F_{\mu\chi}$ attains its maximum in $B_R(p)$ at $x_R$, then
\begin{equation}\label{eq:horizontal gradient at maximum point 2}
0=\nabla^H \ln(\chi F_{\mu\chi})(x_R)=\frac{\nabla^H\chi}{\chi}+\frac{\nabla^H \Phi_{\mu\chi}}{\Phi_{\mu\chi}}+\nu\frac{\nabla^H(\psi\circ u)}{b-\psi\circ u}
\end{equation}
and
\begin{equation}\label{eq:sub laplacian at maximum point 2}
\begin{split}
0\ge&\Delta_H \ln(\chi F_{\mu\chi})(x_R)=\frac{\Delta_H\chi}{\chi}-\frac{\vert\nabla^H\chi\vert^2}{\chi^2}
+\frac{\Delta_H\Phi_{\mu\chi}}{\Phi_{\mu\chi}}\\
&-\frac{\vert\nabla^H\Phi_{\mu\chi}\vert^2}{\Phi_{\mu\chi}^2}+\nu\frac{\Delta_H(\psi\circ u)}{b-\psi\circ u}-\nu\frac{\vert\nabla^H(\psi\circ u)\vert^2}{(b-\psi\circ u)^2}.
\end{split}
\end{equation}
Substituting (\ref{eq:sub laplacian Phi inequality}) into (\ref{eq:sub laplacian at maximum point 2}), we get
\begin{equation}\label{eq:crude step 2 main inequality 1}
\begin{split}
0\ge&\frac{\Delta_H\chi}{\chi}-\frac{\vert\nabla^H\chi\vert^2}{\chi^2}
-\left(\frac{1}{2}+\epsilon_1\right)\frac{\vert\nabla^H\Phi_{\mu\chi}\vert^2}{\Phi_{\mu\chi}^2}-4\bar K e_H(u)\\
&+\nu\frac{\Delta_H(\psi\circ u)}{b-\psi\circ u}-\nu\frac{\vert\nabla^H(\psi\circ u)\vert^2}{(b-\psi\circ u)^2}\\
&+\left(\frac{1}{2}\epsilon_1\eta_{min}+\mu\Delta_H\chi-2\epsilon_1^{-1}\frac{\mu\vert\nabla^H\chi\vert^2}{\chi}-2\mu\chi M_G\right)\frac{e_V(u)}{\Phi_{\mu\chi}} \\
&-\left(2k_1+\frac{C}{\epsilon_2}+2M_G\right)\frac{e_H(u)}{\Phi_{\mu\chi}}.
\end{split}
\end{equation}
Choosing $\epsilon_1=\frac{1}{2\nu}-\frac{1}{4}$, by (\ref{eq:horizontal gradient at maximum point 2}) and Schwarz inequality, we derive that
\begin{equation*}
\begin{split}
&\left(\frac{1}{2}+\epsilon_1\right)\frac{\vert\nabla^H\Phi_{\mu\chi}\vert^2}{\Phi_{\mu\chi}^2} \\
=&\left(\frac{1}{2}+\epsilon_1\right)\left(\frac{\nabla^H\chi}{\chi}+\nu\frac{\nabla^H(\psi\circ u)}{b-\psi\circ u}\right)^2 \\
\le&\left(\frac{1}{2}+\epsilon_1\right)\left(1+\frac{2+\nu}{2-\nu}\right)\frac{\vert\nabla^H\chi\vert^2}{\chi^2}+
\left(\frac{1}{2}+\epsilon_1\right)\left(1+\frac{2-\nu}{2+\nu}\right)\nu^2\frac{\vert\nabla^H(\psi\circ u)\vert^2}{(b-\psi\circ u)^2}\\
=&\frac{2+\nu}{\nu(2-\nu)}\frac{\vert\nabla^H\chi\vert^2}{\chi^2}-\nu\frac{\vert\nabla^H(\psi\circ u)\vert^2}{(b-\psi\circ u)^2}.
\end{split}
\end{equation*}
Therefore, (\ref{eq:crude step 2 main inequality 1}) becomes
\begin{equation}\label{eq:crude step 2 main inequality 2}
\begin{split}
0\ge&\frac{\Delta_H\chi}{\chi}-\left(1+\frac{2+\nu}{\nu(2-\nu)}\right)\frac{\vert\nabla^H\chi\vert^2}{\chi^2}
-4\bar K e_H(u)+\nu\frac{\Delta_H(\psi\circ u)}{b-\psi\circ u}\\
&+\left(\frac{1}{2}\epsilon_1\eta_{min}+\mu\Delta_H\chi-2\epsilon_1^{-1}\frac{\mu\vert\nabla^H\chi\vert^2}{\chi}-2\mu\chi M_G\right)\frac{e_V(u)}{\Phi_{\mu\chi}} \\
&-\left(2k_1+\frac{C}{\epsilon_2}+2M_G\right)\frac{e_H(u)}{\Phi_{\mu\chi}}.
\end{split}
\end{equation}
By (\ref{eq:bound of horizontal gradient chi}), (\ref{eq:lower bound of subelliptic chi}), Lemma \ref{thm:estimate of main auxiliary function} and (\ref{eq:crude step 2 main inequality 2}),
\begin{equation}\label{eq:crude step 2 main inequality 3}
\begin{split}
0\ge&-\frac{C_1}{\chi R}-\left(1+\frac{2+\nu}{\nu(2-\nu)}\right)\frac{C_1}{\chi R^2}
+2\delta e_H(u)\\
&+\left(\frac{1}{2}\epsilon_1\eta_{min}-\mu\frac{C_1}{R}-2\epsilon_1^{-1}\mu\frac{C_1}{R^2}-2\mu\chi M_G\right)\frac{e_V(u)}{\Phi_{\mu\chi}} \\
&-\left(2k_1+\frac{C}{\epsilon_2}+2M_G\right)\frac{e_H(u)}{\Phi_{\mu\chi}}\\
\ge&-\frac{C_7(1+R)}{\chi R^2}
+2\delta e_H(u)\\
&+\left(\frac{1}{2}\epsilon_1\eta_{min}-\mu\frac{C_{\nu}(1+R)}{R^2}-2\mu\chi M_G\right)\frac{e_V(u)}{\Phi_{\mu\chi}} \\
&-\left(2k_1+\frac{C}{\epsilon_2}+2M_G\right)\frac{e_H(u)}{\Phi_{\mu\chi}}
\end{split}
\end{equation}
where $C_7(m,k_1,k_2,\nu)$ is a constant and $\delta$ is given by Lemma \ref{thm:estimate of main auxiliary function}.
Since
\begin{equation*}
e_V(u)=\mu^{-1}\chi^{-1} (\Phi_{\mu\chi}-e_H(u)),
\end{equation*}
from (\ref{eq:crude step 2 main inequality 3}), we obtain
\begin{equation}\label{eq:step 2 main inequality}
\begin{split}
0\ge&-\frac{C_7(1+R)}{\chi R^2}+2\delta e_H(u)+\left(\frac{1}{2}\epsilon_1\eta_{min}-\mu\frac{C_7(1+R)}{R^2}-2 \mu\chi M_G\right)\mu^{-1}\chi^{-1} \\
&+\left(-\frac{1}{2}\epsilon_1\eta_{min}\mu^{-1}\chi^{-1}+\frac{C_7(1+R)}{\chi R^2}-2k_1-\frac{C}{\epsilon_2}\right)\frac{e_H(u)}{\Phi_{\mu\chi}}\\
\ge&\frac{1}{\chi}\left(\frac{1}{2}\epsilon_1\mu^{-1}\eta_{min}-2\frac{C_7(1+R)}{R^2}-2\chi M_G\right) \\
&+\left[2\delta\chi\Phi_{\mu\chi}-\left(\frac{1}{2}\epsilon_1\eta_{min}\mu^{-1}+2k_1+\frac{C}{\epsilon_2}\right)\right]\frac{e_H(u)}{\chi\Phi_{\mu\chi}}.
\end{split}
\end{equation}
Setting
\begin{equation}\label{eq:setting of mu}
\mu^{-1}=\frac{4C_7(1+R)}{\epsilon_1\eta_{min}R^2}+\frac{4\chi M_G}{\epsilon_1\eta_{min}}
\end{equation}
such that
\begin{equation*}
\frac{1}{2}\epsilon_1\mu^{-1}\eta_{min}-2\frac{C_7(1+R)}{R^2}-2\chi M_G>0,
\end{equation*}
then we have
\begin{equation*}
\begin{split}
\chi\Phi_{\mu\chi}&\le\delta^{-1}\left(\frac{C_7(1+R)}{R^2}+k_1+\chi M_G+\frac{C^2}{2\epsilon_1\mu}\right)\\
&\le\delta^{-1}C_8
\end{split}
\end{equation*}
where
\begin{equation*}
C_8=\frac{C_7(1+R)}{R^2}+k_1+M_G+\frac{C^2}{2\epsilon_1\mu}.
\end{equation*}
Therefore, we have
\begin{equation*}
\underset{B_p(2R)}{\text{max}} \chi F_{\mu\chi}\le \frac{\chi\Phi_{\mu\chi}}{(b-\psi\circ u)^{\nu}}(x)\le\frac{ C_8}{\delta(b-\phi(\tau))^{\nu}}.
\end{equation*}
Furthermore, we get
\begin{equation*}
\underset{B_p(R)}{\text{max}} e_H(u)\le b^{\nu}\underset{B_p(R)}{\text{max}} \chi F_{\mu\chi}\le\frac {C_8 b^{\nu}}{\delta(b-\phi(\tau))^{\nu}}.
\end{equation*}
In terms of Lemma \ref{thm:estimate of main auxiliary function}, we know $\delta$, $b$ and $\mu$ all depend on $\bar K$.
By (\ref{eq:setting of mu}), we obtain the following estimate
\begin{equation*}
\begin{split}
\underset{B_p(R)}{\text{max}} e_H(u)&\le C_9\left(\frac{C_7(1+R)}{R^2}+k_1+M_G+\frac{C^2}{2\epsilon_1\mu}\right)\\
&\le C_{10}\left(k_1+M_G+\frac{1+R}{R^2}\right)
\end{split}
\end{equation*}
where $C_9(\bar K)$, $C_{10}(m,k_1,k_2,\bar K)$ are constants.
\end{proof}

\section*{Acknowledgements}
The author is greatly indebted to Professor Yuxin Dong for his kind guidance and constant encouragement. The author also would like to express his thanks to Professor Yibin Ren and Professor Gui Mu for valuable discussions and helpful suggestions. This work is supported by NSFC Grants No.11771087 and No.12171091


\begin{thebibliography}{BGKT18}

\bibitem[AL15]{MR3299291}
Andrei Agrachev and Paul W.~Y. Lee.
\newblock Bishop and {L}aplacian comparison theorems on three-dimensional
  contact sub-{R}iemannian manifolds with symmetry.
\newblock {\em J. Geom. Anal.}, 25(1):512--535, 2015.

\bibitem[Bau16]{MR3587668}
Fabrice Baudoin.
\newblock Sub-{L}aplacians and hypoelliptic operators on totally geodesic
  {R}iemannian foliations.
\newblock In {\em Geometry, analysis and dynamics on sub-{R}iemannian
  manifolds. {V}ol. 1}, EMS Ser. Lect. Math., pages 259--321. Eur. Math. Soc.,
  Z\"{u}rich, 2016.

\bibitem[BDU01]{MR1871387}
Elisabetta Barletta, Sorin Dragomir, and Hajime Urakawa.
\newblock Pseudoharmonic maps from nondegenerate {CR} manifolds to {R}iemannian
  manifolds.
\newblock {\em Indiana Univ. Math. J.}, 50(2):719--746, 2001.

\bibitem[BF15]{baudoin2015logsobolev}
Fabrice Baudoin and Qi~Feng.
\newblock Log-sobolev inequalities on the horizontal path space of a totally
  geodesic foliation, 2015.

\bibitem[BGKT18]{baudoin2018sublaplacian}
Fabrice Baudoin, Erlend Grong, Kazumasa Kuwada, and Anton Thalmaier.
\newblock Sub-laplacian comparison theorems on totally geodesic riemannian
  foliations, 2018.

\bibitem[CDRZ20]{MR4167256}
Tian Chong, Yuxin Dong, Yibin Ren, and Wei Zhang.
\newblock Pseudo-harmonic maps from complete noncompact pseudo-{H}ermitian
  manifolds to regular balls.
\newblock {\em J. Geom. Anal.}, 30(4):3512--3541, 2020.

\bibitem[Che80]{MR573431}
Shiu~Yuen Cheng.
\newblock Liouville theorem for harmonic maps.
\newblock In {\em Geometry of the {L}aplace operator ({P}roc. {S}ympos. {P}ure
  {M}ath., {U}niv. {H}awaii, {H}onolulu, {H}awaii, 1979)}, Proc. Sympos. Pure
  Math., XXXVI, pages 147--151. Amer. Math. Soc., Providence, R.I., 1980.

\bibitem[Che98]{MR1667241}
Qun Chen.
\newblock Harmonic maps with potential from complete manifolds.
\newblock {\em Chinese Sci. Bull.}, 43(21):1780--1786, 1998.

\bibitem[Che99]{MR1680678}
Qun Chen.
\newblock Maximum principles, uniqueness and existence for harmonic maps with
  potential and {L}andau-{L}ifshitz equations.
\newblock {\em Calc. Var. Partial Differential Equations}, 8(2):91--107, 1999.

\bibitem[Cho82]{MR647905}
Hyeong~In Choi.
\newblock On the {L}iouville theorem for harmonic maps.
\newblock {\em Proc. Amer. Math. Soc.}, 85(1):91--94, 1982.

\bibitem[CKLT19]{MR3935276}
Shu-Cheng Chang, Ting-Jung Kuo, Chien Lin, and Jingzhi Tie.
\newblock C{R} sub-{L}aplacian comparison and {L}iouville-type theorem in a
  complete noncompact {S}asakian manifold.
\newblock {\em J. Geom. Anal.}, 29(2):1676--1705, 2019.

\bibitem[HY22]{HY22}
Xin Huang, and Weike Yu.
\newblock A generalization of the Schwarz's lemma for transversally harmonic maps.
\newblock preprint

\bibitem[DLY22]{dong2022subelliptic}
Yuxin Dong, Han Luo, and Weike Yu.
\newblock On subelliptic harmonic maps with potential, 2022.
\newblock

\bibitem[Don21]{MR4236537}
Yuxin Dong.
\newblock Eells-{S}ampson type theorems for subelliptic harmonic maps from
  sub-{R}iemannian manifolds.
\newblock {\em J. Geom. Anal.}, 31(4):3608--3655, 2021.

\bibitem[FR97]{MR1433176}
Ali Fardoun and Andrea Ratto.
\newblock Harmonic maps with potential.
\newblock {\em Calc. Var. Partial Differential Equations}, 5(2):183--197, 1997.

\bibitem[FRR00]{MR1800592}
Ali Fardoun, Andrea Ratto, and Rachid Regbaoui.
\newblock On the heat flow for harmonic maps with potential.
\newblock {\em Ann. Global Anal. Geom.}, 18(6):555--567, 2000.

\bibitem[GW09]{MR2500106}
Detlef Gromoll and Gerard Walschap.
\newblock {\em Metric foliations and curvature}, volume 268 of {\em Progress in
  Mathematics}.
\newblock Birkh\"{a}user Verlag, Basel, 2009.

\bibitem[JX98]{10.2307/117803}
J\"{u}gen Jost and Chao-Jiang Xu.
\newblock Subelliptic harmonic maps.
\newblock {\em Transactions of the American Mathematical Society},
  350(11):4633--4649, 1998.

\bibitem[LL18]{MR3853931}
Paul W.~Y. Lee and Chengbo Li.
\newblock Bishop and {L}aplacian comparison theorems on {S}asakian manifolds.
\newblock {\em Comm. Anal. Geom.}, 26(4):915--954, 2018.

\bibitem[NSW85]{MR793239}
Alexander Nagel, Elias~M. Stein, and Stephen Wainger.
\newblock Balls and metrics defined by vector fields. {I}. {B}asic properties.
\newblock {\em Acta Math.}, 155(1-2):103--147, 1985.

\bibitem[Ren20]{MR4045310}
Yibin Ren.
\newblock Gradient estimate of positive eigenfunctions of sub-{L}aplacian on
  complete pseudo-{H}ermitian manifolds.
\newblock {\em J. Geom. Phys.}, 149:103577, 12, 2020.

\bibitem[Str86]{MR862049}
Robert~S. Strichartz.
\newblock Sub-{R}iemannian geometry.
\newblock {\em J. Differential Geom.}, 24(2):221--263, 1986.

\bibitem[Yau75]{MR431040}
Shing~Tung Yau.
\newblock Harmonic functions on complete {R}iemannian manifolds.
\newblock {\em Comm. Pure Appl. Math.}, 28:201--228, 1975.

\bibitem[Zou21]{zou21}
Wenting Zou.
\newblock Gradient Estimate of Subelliptic Harmonic Maps on Sub-Riemannnian Manifolds
\newblock {\em Advances in Applied Mathematics}, Vol. 10 No. 11 2021. https://doi.org/10.12677/aam.2021.1011416
\end{thebibliography}
\end{document}